\documentclass[a4paper,twoside]{amsart}

\usepackage{hyperref}
\usepackage{amsrefs,mathtools,amssymb}
\renewcommand{\Bbbk}{\mathbf{k}}

\theoremstyle{plain}
\newtheorem{theorem}{Theorem}[section]
\newtheorem{lemma}[theorem]{Lemma}

\newtheorem{proposition}[theorem]{Proposition}

\theoremstyle{definition}
\newtheorem{definition}[theorem]{Definition}

\theoremstyle{remark}
\newtheorem{remark}[theorem]{Remark}

\newcommand{\naturals}{{\mathbb N}}
\newcommand{\integers}{{\mathbb Z}}
\newcommand{\rationals}{{\mathbb Q}}
\newcommand{\reals}{{\mathbb R}}
\newcommand{\finfield}{{\mathbb F_p}}
\newcommand{\abs}[1]{\left|#1\right|}
\newcommand{\onto}{\twoheadrightarrow}
\DeclareMathOperator{\gr}{gr}
\DeclareMathOperator{\Hom}{Hom}

\DeclareMathOperator{\rank}{rank}
\DeclareMathOperator{\Ext}{Ext}

\begin{document}
\title{Right Angled Artin Groups and partial commutation, old and new}

\author[L. Bartholdi]{Laurent Bartholdi}
\email[L.B.]{laurent.bartholdi@math.uni-goettingen.de}
\author[H. H\"arer]{Henrika H\"arer}\thanks{Part of this work is contained in
  the bachelor thesis~\cite{Runde}
of Henrika H\"arer, n\'ee Runde.}
\email[H.H.]{henrikamaria.runde@stud.uni-goettingen.de}
\author[T. Schick]{Thomas Schick}
\email[T.S.]{thomas.schick@math.uni-goettingen.de}
\address[L.B., T.S., H.R.]{Mathematisches Institut, Universit\"at G{\"o}ttingen, Germany}
\address[L.B.]{\'Ecole Normale Sup\'erieure, Lyon, France}
\date{May 1, 2020}

\begin{abstract}
  We systematically treat algebraic objects with free partially commuting
  generators and give short and modern proofs of the various relations between
  them. These objects include right angled Artin groups, polynomial rings, Lie
  algebras, and restricted Lie algebras in partially commuting free
  generators. In particular, we compute the $p$-central and exponent-$p$ series of all right angled Artin groups, and compute the dimensions of their subquotients. We also describe their associated Lie algebras, and relate them to the cohomology ring of the group as well as to polynomial and power series rings in partially commuting variables. We finally show how the growth series of these various objects are related to each other.
\end{abstract}

\maketitle

\section{Introduction}

Right angled Artin groups (RAAGs) are a prominent geometric/combinatorial
class of groups. Originally introduced as ``partially commuting free groups'',
they interpolate in an interesting way between free groups and free abelian
groups. Of particular
interest are several additional algebraic objects which are canonically
coming along and are closely related to the structure of the RAAGs, in
particular (graded) Lie algebras and polynomial rings, both in free partially
commuting generators. The purpose of this article is to give a complete description of
many relevant properties and relations, offering modern and accessible
proofs. Many of the
results quoted below appear already in other sources, though the computation
of the exponent-$p$ and lower $p$-central series is new.

\subsection{The actors}

Let $\Gamma$ be an undirected graph, with vertex set $V$
and edge set $E$
(consisting of $2$-element subsets of $V$). The right
angled Artin group (RAAG) $A_\Gamma$ associated with $\Gamma$ is the group
defined in terms of generators and relations as
\[A_\Gamma=\langle V\mid v w=w v\text{ whenever } \{v,w\}\in E\rangle.\]

The purpose of this note is to describe classical subgroup series in
$A_\Gamma$ such as the lower-central and $p$-lower-central series, and relate
them to other algebraic objects defined in terms of $\Gamma$ as follows. 

Let $\Bbbk$ be a commutative ring. We define unital associative $\Bbbk$-algebras
\[\begin{split}R_\Gamma&=\langle V\mid v w=w v\text{ whenever }\{v,w\}\in E\rangle,\\
S_\Gamma&=\langle V\mid v^2=0, \,v w=-w v\,\forall v,w\in V,\text { and }v
w=0\text{ whenever }\{v,w\}\not\in E\rangle.\end{split}\]
Note that $R_\Gamma$ is the familiar algebra of
polynomials in partially commuting variables, and similarly $S_\Gamma$ can be
considered as an exterior algebra in partially commuting variables.

Observe that $R_\Gamma$ and $S_\Gamma$ are \emph{graded} algebras with $\deg(v)=1$ for
all $v\in V$. Therefore, they admit a natural topology, in which basic
neighbourhoods of $0$ (say in $R_\Gamma$) are spans of the set of all
monomials of degree $\ge n$. We define
\begin{equation*}
\overline{R_\Gamma} = \text{the completion
of $R_\Gamma$ in this topology}.
\end{equation*}
Just as $R_\Gamma$ is a
non-commutative polynomial algebra, $\overline{R_\Gamma}$ is an
algebra of power series in partially commuting variables.

We also define a Lie algebra over $\Bbbk$,
\[L_\Gamma=\langle V\mid [v,w]=0\text{ whenever }\{v,w\}\in E\rangle,\]
and, if $\Bbbk$ is an algebra over $\finfield$, a \emph{restricted} Lie algebra (see Section~\ref{ss:hopf} for a review of restricted Lie algebras)
\[L_{\Gamma,p}=\langle V\mid [v,w]=0\text{ whenever }\{v,w\}\in E\rangle_p.\]

Let us have a look at the extreme cases.
\begin{enumerate}
\item If $\Gamma$ is the complete graph on $d$ vertices then
  $A_\Gamma\cong\mathbb Z^d$, $R_\Gamma$ is the polynomial algebra in $d$
  variables $\Bbbk[X_1,\dots,X_d]$, $S_\Gamma$ is the Grassmann algebra
  $\bigwedge^*(\Bbbk^d)$, and $L_\Gamma\cong\Bbbk^d$ with trivial bracket. 
\item If
  $\Gamma$ is the empty graph on $d$ vertices then $A_\Gamma$ is the free
  group $F_d$, $R_\Gamma$ is the free associative algebra on $d$ generators,
  $S_\Gamma\cong\Bbbk\cdot1\oplus\Bbbk^d$ with trivial multiplication except
  $1\cdot x=x$, and $L_\Gamma$ is the free Lie algebra on $d$ generators; for
  more details see Section~\ref{ss:examples}.
\end{enumerate}

\subsection{Subgroup series}

Let $G$ be any discrete group, and let $\rho\colon G\to R^\times$ be a
representation of $G$ in an associative augmented $\Bbbk$-algebra $R$ with
augmentation ideal $\varpi$ (namely, an algebra equipped with an epimorphism
to $\Bbbk$ with kernel $\varpi$). With this representation is associated a
natural sequence of subgroups, called \emph{generalized dimension subgroups},
\[\delta_{n,\rho}\coloneqq\rho^{-1}(1+\varpi^n)=\ker(G\to (R/\varpi^n)^\times).\]
In case $R=\Bbbk G$ and $\rho$ is the regular representation, we write $\delta_{n,\Bbbk G}$ for $\delta_{n,\rho}$.

In addition, there are classical subgroup series, defined intrinsically within
$G$: 
\begin{itemize}
\item the \emph{lower central series} $(\gamma_n)$ given by $\gamma_1=G$ and $\gamma_n=[\gamma_{n-1},G]$;
\item the \emph{rational lower central series} $\gamma_{n,0}=\{g\in G\mid g^k\in\gamma_n\text{ for some }k\neq0\}$;
\item for a prime $p$ fixed throughout the discussion, the \emph{exponent-$p$ central series} $\lambda_{n,p}$ given by $\lambda_{1,p}=G$ and $\lambda_{n,p}=[\lambda_{n-1,p},G]\lambda_{n-1,p}^p$, or more directly $\lambda_{n,p}=\prod_{m+i\ge n}\gamma_m^{p^i}$;
\item again for a prime $p$ fixed throughout the discussion, the \emph{Brauer-Jennings-Lazard-Zassenhaus} series~\cites{Zassenhaus,Jennings:1941,Lazard:1954}, also called \emph{$p$-dimension} or \emph{$p$-central} series, given by $\gamma_{1,p}=G$ and $\gamma_{n,p}=[\gamma_{n-1,p},G]\gamma_{\lceil n/p\rceil,p}^p$, or more directly $\gamma_{n,p}=\prod_{m p^i\ge n}\gamma_m^{p^i}$.
\end{itemize}

All these series are \emph{central}, meaning that $\gamma_n/\gamma_{n+1}$ belongs to
the center of $G/\gamma_{n+1}$, etc. We moreover have
$[\gamma_m,\gamma_n]\subseteq\gamma_{m+n}$, etc. A classical
consequence~\cite{MagnusKarrassSolitar}*{Section 5.3} is that
$\bigoplus_{n\ge1}\gamma_n/\gamma_{n+1}$, etc., are graded Lie algebras over
$\integers$. The addition is induced by the group multiplication and the Lie
bracket is induced by the commutator.

The groups  $\gamma_{n,0}$ enjoys the extra property that
$\gamma_{n,0}/\gamma_{n+1,0}$ is torsion-free (and it is the fastest
descending central series with this property), so
$\bigoplus_{n\ge1}\gamma_{n,0}/\gamma_{n+1,0}$ is $\integers$-free. In
particular, if $\gamma_n/\gamma_{n+1}$ is torsion free for each $n$, then
$\gamma_{n,0}=\gamma_n$ for each $n$. 

We have $\lambda_{n,p}^p\subseteq\lambda_{n+1,p}$ so
$\bigoplus_{n\ge1}\lambda_{n,p}/\lambda_{n+1,p}$ is an elementary abelian
$p$-group. Similarly, $\gamma_{n,p}^p\subseteq\gamma_{n p,p}$. Furthermore, these series are fastest descending under these
requirements. It is now
classical~\cite{Zassenhaus} that
$\bigoplus_{n\ge1}\gamma_{n,p}/\gamma_{n+1,p}$ is a restricted Lie algebra
over $\finfield$. The additional, ``$p$-power'' operation as part of the \emph{restricted}
Lie algebra structure is induced by the $p$-power operation in the
group. 

Classical results identify $\delta_{n,\Bbbk G}$ with some of the above
series in case $\Bbbk$ is a field: we have
$\delta_{n,\Bbbk G}=\gamma_{n,p}$ where $p\ge0$ is the characteristic of
$\Bbbk$ \cites{Hall,Jennings:1941,Jennings:1955}. However, for general $G$, the identification of
$\delta_{n,\integers G}$ is a fundamental open problem of group
theory.

\subsection{Results}
We consider the series defined above for the group $A_\Gamma$. The
main purpose of this text is to exhibit numerous relations between
these algebraic objects; detailed definitions and proofs will be given
in subsequent sections. The main tool is an extension to $A_\Gamma$
of Magnus's work on the free group~\cite[Section
5]{MagnusKarrassSolitar}, embedding it into the units of the free
non-commuting power series 
ring. This extension seems first considered in~\cite{Droms:1983}. 

Recall that a commutative ring $\Bbbk$ is fixed. Denote by $\varpi$ the
augmentation ideal of $R_\Gamma$ (i.e.~the ideal of polynomials in partially
commuting variables with zero constant term), and by $\varpi(A_\Gamma)$ the
augmentation ideal of $\Bbbk A_\Gamma$.
\begin{theorem}[Augmentation ideals]\label{thm:dim}
  For all $n$ we have 
  \[\varpi(A_\Gamma)^n/\varpi(A_\Gamma)^{n+1}\cong \varpi^n/\varpi^{n+1}.\]
\end{theorem}

We remind the reader that Koszul algebras are a particular kind of associative algebras
(see~\cite{Priddy:1970} or Section~\ref{ss:coh}) for which a ``small''
projective resolution may easily be computed. Moreover, there is the important
concept of Koszul duality. We obtain the following results, which for
$\Bbbk=\rationals$ already appear in~\cite{PapadimaSuciu:2006}.
\begin{theorem}[Group cohomology]\label{thm:coh}
  Let $S^1$ be the circle with base point $*$. The following subspace of the torus $(S^1)^V$ is a classifying space for $A_\Gamma$:
\begin{equation}\label{eq:classifying_space}
X_\Gamma=\bigcup_{C\subseteq V\text{ a clique}}(S^1)^C\times\{*\}^{V\setminus
  C}.
\end{equation}

  We have $H^*(A_\Gamma;\Bbbk)=S_\Gamma$.

  The rings $R_\Gamma$ and $S_\Gamma$ are Koszul algebras, and Koszul duals to each other: $(S_\Gamma)^!=R_\Gamma$.
\end{theorem}

\begingroup
\renewcommand\thetheorem{\arabic{section}.\arabic{theorem}i}
\begin{theorem}[Central series and dimension subgroups]\label{thm:cs1}
  We have
  \begin{equation*}
\gamma_n=\gamma_{n,0} \quad\text{ and}\quad\bigcap_{n\in\integers} \gamma_{n,0}= \bigcap_{n\in\integers}\gamma_{n,p}=\{1\}.
\end{equation*}
In particular, $A_\Gamma$ is finitely generated and residually torsion-free
nilpotent, so (by~\cite{Gruenberg:1957}*{Theorem~2.1})
$A_\Gamma$ is also a residually finite $p$-group for every $p$.
\end{theorem}

\renewcommand\thetheorem{\arabic{section}.\arabic{theorem}ii}
\addtocounter{theorem}{-1}
\begin{theorem}[Central series and dimension subgroups]\label{thm:cs2}
  There is a faithful representation
  \begin{equation*}
\mu\colon A_\Gamma\to\overline{R_\Gamma}^\times;
\quad v\mapsto1+v\quad\text{for } v\in V.
\end{equation*}
The corresponding generalized dimension subgroups satisfy
  \[\delta_{n,\mu}=\begin{cases}
      \gamma_{n,0} & \text{ if $\Bbbk$ has characteristic }0,\\
      \gamma_{n,p} & \text{  if $\Bbbk$ has characteristic }p.
    \end{cases}\]
    Together with Theorem~\ref{thm:dim} we obtain an isomorphism of filtered associative\footnote{but not Hopf algebras; see Theorem~\ref{thm:malcev} below!} $\Bbbk$-algebras
  \[\overline{\Bbbk A_\Gamma}\coloneqq\lim (\Bbbk A_\Gamma/\varpi(G)^n)\xrightarrow{\cong} \lim (R_\Gamma/\varpi^n)=\overline{R_\Gamma}.\]
  In particular, the classical dimension subgroups coincide:
  \begin{equation*}
    \delta_{n,\Bbbk A_\Gamma} = \delta_{n,\mu}.
  \end{equation*}
\end{theorem}
\endgroup

The Lie algebras $L_\Gamma$ and $L_{\Gamma,p}$ are tightly connected to their associative counterparts:
\begingroup
\renewcommand\thetheorem{\arabic{section}.\arabic{theorem}i}
\begin{theorem}[Lie algebras]\label{thm:lie1}
  The algebra $R_\Gamma$ is a Hopf algebra. If the ring $\Bbbk$ is a $\integers$-free module then we have
  \begin{equation}
  L_\Gamma\cong\text{Primitives}(R_\Gamma) \quad\text{and}\quad R_\Gamma\cong
  U(L_\Gamma),\label{eq:char0}
  \end{equation}
  the universal enveloping algebra of $L_\Gamma$, while if $\Bbbk$ is an $\finfield$-algebra then
  \begin{equation}
  L_{\Gamma,p}\cong\text{Primitives}(R_\Gamma) \quad\text{and}\quad
  R_\Gamma\cong U_p(L_{\Gamma,p}),\label{eq:charp}
  \end{equation}
  the $p$-universal enveloping algebra of $L_{\Gamma,p}$. The Lie algebra cohomology of $L_\Gamma$ is
  \[H^*(L_\Gamma;\Bbbk)\cong S_\Gamma.\]

  All the above isomorphisms are natural, in the sense that they are induced
  by the identity map $V\to V$, and therefore compatible with homomorphisms
  induced by a map of graphs $V\to V'$.
\end{theorem}

The Lie algebra associated with the lower central series was already determined in~\cite{DuchampKrob1} as $L_\Gamma$. We extend this result as follows:
\renewcommand\thetheorem{\arabic{section}.\arabic{theorem}ii}
\addtocounter{theorem}{-1}
\begin{theorem}[Lie algebras]\label{thm:lie2}
  For any ring $\Bbbk$, we have, as Lie algebras,
  \[L_\Gamma\cong\bigoplus_{n\ge1}(\gamma_n/\gamma_{n+1})\otimes_\integers\Bbbk.\]
  If $\finfield\subseteq\Bbbk$ then as restricted Lie algebras
  \begin{equation*}
    L_{\Gamma,p}\cong
  \bigoplus_{n\ge1}(\gamma_{n,p}/\gamma_{n+1,p})\otimes_\integers\Bbbk.
\end{equation*}  
  If $\finfield\subseteq\Bbbk$ and $p\ge3$ then with $\Bbbk[\pi]$ the polynomial ring in one degree-$1$ variable $\pi$
  \begin{equation*}
L_\Gamma\otimes_\Bbbk\Bbbk[\pi]\cong\bigoplus_{n\ge1}(\lambda_{n,p}/\lambda_{n+1,p})\otimes_\integers\Bbbk;
\end{equation*}
 under that isomorphism, multiplication by $\pi$ corresponds to the map induced by $\lambda_{n,p}\ni g\mapsto g^p\in \lambda_{n+1,p}$.
  
  All the above isomorphisms are natural, in the sense that they are induced
  by the identity map $V\to V$, and therefore compatible with homomorphisms
  induced by a map of graphs $V\to V'$.
\end{theorem}
\endgroup

For a graded algebra $R=\bigoplus_{n\ge0}R_n$ over $\Bbbk$ such that each $R_n$ is a finitely generated free $\Bbbk$-module, recall that its \emph{Poincar\'e series} is the power series 
\[\Phi_R(t)=\sum_{n\ge0}\rank(R_n)t^n.\]
For a group $G=\langle X\rangle$, its \emph{growth series} is
$\Phi_G(t)=\sum_{g\in G}t^{|g|}$, with $|g|$ denoting the word length of $g\in
G$ (word length and growth series depend on the fixed generating set $X$). The
first two claims of the following result appear in~\cite{DuchampKrob2}:
\begin{theorem}[Poincar\'e and growth series]\label{thm:growth}
  The Poincar\'e series of $S_\Gamma$ is
  \begin{equation*}
\Phi_{S_\Gamma}(t)=\sum_{n=0}^{\abs{V}}c_n(\Gamma)t^n,
\end{equation*}
where $c_n(\Gamma)$ denotes the number of cliques of size $n$ (i.e.~complete subgraphs of $\Gamma$ with $n$ vertices).

  The Poincar\'e series of $R_\Gamma$ and $S_\Gamma$ are connected by the relation
  \[\Phi_{R_\Gamma}(t)\cdot\Phi_{S_\Gamma}(-t)=1,\]
  and the growth series of $A_\Gamma$ is
  \[\Phi_{A_\Gamma}(t) = \Phi_{R_\Gamma}\bigg(\frac{2t}{1+t}\bigg).\]
\end{theorem}

In our next result, originally appearing
in~\cite{KapovichMillson:1998}*{Theorem~16.10}, we determine the
Malcev completion of $A_\Gamma$. We refer
to~\cites{Malcev:1949,Quillen:1969} and the more
recent~\cite{PapadimaSuciu:2004} for a review of this construction.

\begin{theorem}[Malcev completions]\label{thm:malcev}
  Assume $\Bbbk=\rationals$. There is then an isomorphism
  $\mu_{\exp}\colon\overline{R_\Gamma}\to\overline{\rationals
    A_\Gamma}$ of filtered, complete Hopf algebras; via this isomorphism,
  $\overline{L_\Gamma}$ is the Malcev Lie algebra of $A_\Gamma$, and
  the Malcev completion of $A_\Gamma$ is given on generators by
  \[A_\Gamma\to\exp(L_\Gamma)\subset\overline{R_\Gamma}\text{ via the classical power series }v\mapsto\sum_{n\ge0} \frac{v^n}{n!}\;\;\forall v\in V.\]
\end{theorem}

We also show the following related result on formality in the sense of
rational homotopy theory; see Section~\ref{ss:malcev} for a review of the notion. 

\begin{theorem}[Formality]
  The classifying space $X_\Gamma$ of $A_\Gamma$ of \eqref{eq:classifying_space}is formal.
\end{theorem}

\subsection{Examples and illustrations}\label{ss:examples}

Let us consider, as sketched in the Introduction, the two extreme cases of graphs $\Gamma$, the complete and empty graphs.

If $\Gamma$ is the complete graph on $V$, then $A_\Gamma$ is free Abelian with basis $V$, and $R_\Gamma$ is a usual polynomial algebra in variables $V$. The standard Koszul complex is given by the exterior algebra $S_\Gamma=\bigwedge^*(V)$, and coincides with the cohomology ring of $A_\Gamma=\integers V$. The classifying space $X_\Gamma$ is the usual torus $(S^1)^V$. The exponent-$p$ central series satisfies $\lambda_{n,p}=p^{n-1}\integers V$, and the $p$-dimension series satisfies $\gamma_{n,p}=p^i\integers V$ whenever $p^{i-1}<n\le p^i$. The growth series are readily computed as
\[\Phi_{A_\Gamma}(t)=\bigg(\frac{1+t}{1-t}\bigg)^{\abs{V}},\quad\Phi_{R_\Gamma}(t)=\bigg(\frac{1}{1-t}\bigg)^{\abs{V}},\quad\Phi_{S_\Gamma}(t)=\big(1+t\big)^{\abs{V}}.\]

If, on the other hand, $\Gamma$ is the empty graph on $V$, then $A_\Gamma$ is
free with basis $V$, and $R_\Gamma$ is a polynomial algebra in non-commuting
variables $V$. The algebra $S_\Gamma$ is reduced to $\Bbbk\oplus\Bbbk V$ with
$V^2=0$, and coincides with the cohomology ring of $A_\Gamma$. The classifying
space $X_\Gamma$ is a wedge of $\abs{V}$ circles. The Lie algebras $L_\Gamma$
and the restricted Lie algebra $L_{\Gamma,p}$ are free. The growth series are
readily computed as 
\[\Phi_{A_\Gamma}(t)=\frac{1+t}{1-(2\abs{V}-1)t},\quad\Phi_{R_\Gamma}(t)=\frac{1}{1-\abs{V} t},\quad\Phi_{S_\Gamma}(t)=1+\abs{V} t.\]

These results can be seen as special cases of the following
constructions. If $\Gamma$ is the disjoint union of two graphs
$\Gamma_1\sqcup\Gamma_2$, then $A_\Gamma=A_{\Gamma_1}*A_{\Gamma_2}$ is a
free product of groups, and similarly $L_\Gamma$ and $R_\Gamma$ are free
products in their respective categories, and
$S_\Gamma=S_{\Gamma_1}\oplus S_{\Gamma_2}/(1\oplus0=0\oplus1)$. The
space $X_\Gamma$ is the wedge (one-point union) of $X_{\Gamma_1}$ and
$X_{\Gamma_2}$, and the growth series of $A_\Gamma,R_\Gamma,S_\Gamma$
may be easily be deduced from those of $A_{\Gamma_1}$, $A_{\Gamma_2}$,
etc.:
\begin{align*}
  1-\frac1{\Phi_{A_\Gamma}}&=\Big(1-\frac1{\Phi_{A_{\Gamma_1}}}\Big)+\Big(1-\frac1{\Phi_{A_{\Gamma_2}}}\Big),\\
  1-\frac1{\Phi_{R_\Gamma}}&=\Big(1-\frac1{\Phi_{R_{\Gamma_1}}}\Big)+\Big(1-\frac1{\Phi_{R_{\Gamma_2}}}\Big),\\
  1-\Phi_{S_\Gamma}&=\big(1-\Phi_{S_{\Gamma_1}}\big)+\big(1-\Phi_{A_{\Gamma_2}}\big).
\end{align*}

If $\Gamma$ is the join of two graphs $\Gamma_1$ and $\Gamma_2$, namely
the graph obtained from $\Gamma_1\sqcup\Gamma_2$ by adding all edges
between $\Gamma_1$ and $\Gamma_2$, then
$A_\Gamma=A_{\Gamma_1}\times A_{\Gamma_2}$ is a direct product, and
similarly $L_\Gamma=L_{\Gamma_1}\times L_{\Gamma_2}$ and
$R_\Gamma=R_{\Gamma_1}\otimes R_{\Gamma_2}$, while $S_\Gamma$ is
$S_{\Gamma_1}\otimes S_{\Gamma_2}$ qua $\Bbbk$-module, with product
$(a\otimes b)(c\otimes d)=(-1)^{\deg(b)\deg(c)}(a c\otimes b d)$. The
classifying space is $X_\Gamma=X_{\Gamma_1}\times X_{\Gamma_2}$, and
the growth series $\Phi_{A_\Gamma}$, $\Phi_{R_\Gamma}$ and
$\Phi_{S_\Gamma}$ behave multiplicatively:
\begin{xalignat*}{3}
  \Phi_{A_\Gamma}&=\Phi_{A_{\Gamma_1}}\cdot\Phi_{A_{\Gamma_2}},&
  \Phi_{R_\Gamma}&=\Phi_{R_{\Gamma_1}}\cdot\Phi_{R_{\Gamma_2}},&
  \Phi_{S_\Gamma}&=\Phi_{S_{\Gamma_1}}\cdot\Phi_{S_{\Gamma_2}}.
\end{xalignat*}

Finally, all the objects constructed are functorial, in the sense that graph morphisms induce maps between the corresponding objects: if $\Gamma,\Gamma'$ are graphs and $f\colon\Gamma\to\Gamma'$ is a map from the vertex set of $\Gamma$ to that of $\Gamma'$ sending edges of $\Gamma$ to edges of $\Gamma'$, then there is an induced group homomorphism $f_*\colon A_\Gamma\to A_{\Gamma'}$, ring homomorphism $R_\Gamma\to R_{\Gamma'}$ and $S_{\Gamma'}\to S_\Gamma$ (note the direction!), etc. Furthermore, if $f$ is injective and full (meaning that $\{f(v),f(w)\}$ is an edge in $\Gamma'$ precisely when $\{v,w\}$ is an edge in $\Gamma$) then the corresponding group and ring homomorphisms are injective.

\subsection{Structure of the article}
The article introduces and relies on quite a number of different
concepts (Hopf algebras, the Magnus map, \dots). These are introduced
one after the other in the following sections. In particular,
Section~\ref{ss:hopf} collects some basic information about
(restricted) Lie algebras and Hopf algebras which we use as technical
tools; we prove the first part of Theorem~\ref{thm:lie1} in it.

Section~\ref{ss:magnus} introduces the Magnus map, which embeds the
group $A_\Gamma$ into the units of the partially commuting power
series ring $\overline{R_\Gamma}$. We show that this map is compatible
with the central series filtrations (and dimension series
filtrations). The explicit knowledge of the structure of the power
series ring can be transferred to $A_\Gamma$ to give the desired
information about the latter. We also prove Theorem~\ref{thm:dim} in it.

We next introduce cohomological notions in Section~\ref{ss:coh}, and
use them to prove Theorem~\ref{thm:coh}.

We study central series in more depth in Section~\ref{ss:central}, and
prove there the first, easy part of Theorem~\ref{thm:cs1}. The second
part requires more knowledge on the Lie algebras $L_\Gamma$, which we
describe in Section~\ref{ss:lie}; we prove Theorems~\ref{thm:cs2}
and~\ref{thm:lie2} there. We also complete there the proof of
Theorem~\ref{thm:lie1} that pertains to Lie algebra cohomology.

Finally Section~\ref{ss:growth} proves Theorem~\ref{thm:growth} and
Section~\ref{ss:malcev} proves Theorem~\ref{thm:malcev}.  We apologize
to the reader if the proofs are not given in strictly linear order; we
found it preferable to prove individual statements of the main results
where the appropriate tools were introduced.


\section{Lie and Hopf algebras}\label{ss:hopf}
We first recall from~\cite{Jacobson:1941} that a restricted Lie
algebra over $\Bbbk$, in characteristic $p$, is a Lie algebra equipped with an
extra operation, written $x\mapsto x^{[p]}$, called the
\emph{$p$-mapping} and subject to the following axioms, where we use the
standard multi-commutator convention $[x,y,z]=[x,[y,z]]$, etc. For all $x,y$
in the Lie algebra and $\alpha\in\Bbbk$,
\begin{equation*}
[y,x^{[p]}]=[y,x,\dots,x]  \quad\text{($p$ factors `$x$')};\qquad(\alpha x)^{[p]}=\alpha^p x^{[p]};
\end{equation*}
\begin{equation*}
(x+y)^{[p]}=x^{[p]}+y^{[p]}+\sum_{i=1}^{p-1}s_i(x,y)
\end{equation*}
for the Lie
expressions $s_i(X,Y)$ defined by

\begin{equation*}
\frac{d}{dt} [X,t X+Y,\dots,t X+Y]=\sum s_i(X,Y) t^i \text{  with $p-1$ factors }
`t X+Y'.
\end{equation*}
For example, if $p=2$ then $s_1(X,Y)=[X,Y]$, and if $p=3$
then $s_1(X,Y)=[Y,X,Y]$ and $s_2(X,Y)=[X,Y,X]$. 

We adopt the convention that, in characteristic $0$, every Lie algebra is restricted with trivial $p$-mapping. This way, from now on we can uniformly work with restricted Lie algebras.

Recall that every restricted Lie algebra $L$ has a \emph{restricted
  universal enveloping algebra}, a unital associative algebra $U_p(L)$
equipped with a map of restricted Lie algebras $L\to U_p(L)$,
universal with respect to this property. The Lie bracket in $L$ is
identified with the commutator $[x,y]=x y-y x$, and the $p$-mapping in
$L$ is identified with the $p$-power operation in $U_p(L)$. The map $L\to
U_p(L)$ is injective.

Recall next that a Hopf algebra is an associative algebra $R$ equipped
with additional structure, in particular an \emph{augmentation}
$\varepsilon\colon R\to\Bbbk$ and a \emph{coproduct}
$\Delta\colon R\to R\otimes R$ which are algebra homomorphisms, and an
\emph{antipode} $S\colon R\to R$ which is an algebra antihomomorphism,
subject to some axioms that we shall not need;
see~\cite{Sweedler:1969}.

We will use the following classical facts, see~\cite{Serre:1964}*{Theorem~III.5.4 and Exercise~2}.
\begin{proposition}\label{prop:Hopf_background}
  The (restricted) universal enveloping algebra $U(L)$, respectively
  $U_p(L)$, is a Hopf algebra. The augmentation, coproduct and antipode are
  given by
  \begin{equation*}
    \varepsilon\colon U(L)\onto U(L)/\langle
    L\rangle\xrightarrow{\cong}\Bbbk;\qquad \Delta(x)=x\otimes1+1\otimes x;\qquad
    S(x)=-x \;\;\forall x\in L.
  \end{equation*}

In a Hopf algebra $H$, call $x\in H$ \emph{primitive} if
  $\Delta(x)=x\otimes1+1\otimes x$; the primitive elements of $H$ form a Lie
  subalgebra $P$ of $H$. If the ring $\Bbbk$ is a $\integers$-free module,
  then the primitive elements in $U(L)$ coincide with $L$, while if $L$ is
  restricted and $\Bbbk$ is $p$-torsion then the primitive elements in
  $U_p(L)$ coincide with $L$;
\end{proposition}

  If a (restricted) Lie algebra over $\Bbbk$ is given by a (restricted) Lie
  algebra presentation, then by the universal property the same presentation,
  now as a presentation of algebras over $\Bbbk$, defines its (restricted)
  universal enveloping algebra. In particular, $R_\Gamma$ is the (restricted)
  enveloping algebra of $L_\Gamma$ or $L_{\Gamma,p}$, respectively.

\begin{proof}[Proof of Theorem~\ref{thm:lie1}]
 As a universal enveloping
  algebra,    $R_\Gamma=U(L_\Gamma)$ is by Proposition
  \ref{prop:Hopf_background}  a Hopf algebra 
  (this also appears in~\cite{Schmitt}), and  its Lie subalgebra of
  primitive elements $P$ is equal to $L_\Gamma$ or $L_{\Gamma,p}$,
  when considered as subset of $R_\Gamma$ in the obvious way.
\end{proof}

We note for later use the following standard constructions,
see also~\cite{Quillen:1968}.
\begin{proposition}\label{prop:Hopf_group}
If $G$ is a group then the group ring $\Bbbk G$ is a Hopf algebra with
augmentation, coproduct and antipode given as follows:
\begin{equation*}
\varepsilon\colon\Bbbk G\to\Bbbk\text{ induced by the
  map }G\to 1;\qquad \Delta(g)=g\otimes g; \qquad S(g)=g^{-1}\;\forall g\in
G.
\end{equation*}

  Furthermore, if $H$ is a Hopf algebra and $\varpi$ denotes its augmentation
  ideal $\ker(\varepsilon)$, then $\bigoplus_{n\ge0}\varpi^n/\varpi^{n+1}$ is
  naturally a graded Hopf algebra.
\end{proposition}

\section{The Magnus map}\label{ss:magnus}

\subsection{Filtrations and gradings}

We first recall that, since the relations of $R_\Gamma$ and $S_\Gamma$ are homogeneous, these rings are naturally graded by setting $deg(v)=1$ for all $v\in V$. We view $R_\Gamma$ as a ring of polynomials in partially commuting variables $v\in V$.

Let us consider the augmentation ideal $\varpi=\langle V\rangle$ in $R_\Gamma$. It consists of all polynomials without constant term. Note that $\varpi^n$ then consists of all polynomials with no terms of degree $<n$. We define a topology on $R_\Gamma$ by declaring the sets $\varpi^n$ to form a basis of neighbourhoods of $0$, and let $\overline{R_\Gamma}$ be the completion of $R_\Gamma$ in this topology. We thus have
\[R_\Gamma\cong\bigoplus_{n\ge0}\varpi^n/\varpi^{n+1},\qquad\overline{R_\Gamma}\cong\prod_{n\ge0}\varpi^n/\varpi^{n+1}.\]

We write $\overline\varpi$ for the closure of $\varpi$ in $\overline{R_\Gamma}$. It consists of all power series with vanishing constant term, and similarly $\overline\varpi^n$ consists of the power series with no terms of degree $<n$.

For comparison, consider the group ring $\Bbbk A_\Gamma$, and let $\varpi(A_\Gamma)$ denote the augmentation ideal of $\Bbbk A_\Gamma$; it is the ideal 
\[\langle g-1\mid g\in A_\Gamma\rangle=\langle v-1\mid v\in V\rangle.\]
We topologize $\Bbbk A_\Gamma$ by declaring the $\varpi(A_\Gamma)^n$ to form a basis of neighbourhoods of the identity, and let $\overline{\Bbbk A_\Gamma}$ denote the corresponding completion. Moreover, let $\gr(\Bbbk A_\Gamma)\coloneqq\bigoplus_{n\ge0}\varpi(A_\Gamma)^n/\varpi(A_\Gamma)^{n+1}$ be the associated graded algebra. We isolate the main ingredient of Theorem~\ref{thm:dim}:
\begin{lemma}\label{lem:dim}
  We have $R_\Gamma\cong \gr(\Bbbk A_\Gamma)$ as graded algebras via the natural map 
  \[\alpha\colon R_\Gamma \to  \gr(\Bbbk A_\Gamma); v_j \mapsto [(v_j-1)]\text{ for }v_j\in V.\]
\end{lemma}
\begin{proof}
  The isomorphism between the degree-$n$ subspace of $R_\Gamma$ and
  $\varpi(A_\Gamma)^n/\varpi(A_\Gamma)^{n+1}$ can be proven by
  elementary considerations, since
  $\varpi(A_\Gamma)^n/\varpi(A_\Gamma)^{n+1}$ is generated by
  expressions $(v_1-1)\cdots(v_n-1)$.

  However, here is a somewhat more elegant
  shortcut: as 
  we noted in Propositions \ref{prop:Hopf_background} and
  \ref{prop:Hopf_group}, $\Bbbk A_\Gamma$, 
  $\gr(\Bbbk A_\Gamma)$, and $R_\Gamma$ are all cocommutative Hopf
  algebras, with coproduct induced respectively by
  $\Delta(g)=g\otimes g$, by
  $\Delta([g-1])=[(g-1)\otimes1+1\otimes(g-1)]$ for $g\in A_\Gamma$
  and by $\Delta(v)=v\otimes 1+1\otimes v$ for $v\in V$.

  The map $\alpha\colon R_\Gamma\to \gr(\Bbbk A_\Gamma)$  is a well defined map
  of  unital graded algebras because the defining commutation property for the
  $v_j$   in $R_\Gamma$ is satisfied for their images, and all these elements
  are of degree $1$. Moreover, we see that this map is a map of Hopf algebras.
  
 Finally, $\alpha$ is an isomorphism when restricted
  to the   degree $1$ subspaces, since
  \begin{equation*}
\varpi/\varpi^2 \cong\Bbbk
  V  \cong(A_\Gamma/[A_\Gamma,A_\Gamma])\otimes\Bbbk \cong \varpi(A_\Gamma)/\varpi(A_\Gamma)^2.
\end{equation*}
Here, the last isomorphism is the standard
  isomorphism of the first group homology
  $H^1(A_\Gamma;\Bbbk)=A_\Gamma/[A_\Gamma,A_\Gamma]\otimes \Bbbk$ as
  $\varpi(A_\Gamma)/\varpi(A_\Gamma)^2$. We conclude
  by~\cite{MilnorMoore:1965}*{Theorems~5.18 and~6.11} that $\alpha$ is
  an isomorphism: it is a map between cocommutative Hopf algebras both
  generated as algebras in degree $1$ and the map
  is an isomorphism in degree $1$. This shortcut already appears
  in~\cite{Quillen:1968}.
\end{proof}

\begin{remark}
  An alternative proof of Lemma~\ref{lem:dim} was kindly suggested to
  us by Jacques Darn\'e: there are natural maps
  \begin{equation*}
  A_\Gamma\supset V\to\overline{R_\Gamma}\quad\text{ and }\quad R_\Gamma\supset V\to\overline{\Bbbk A_\Gamma},
\end{equation*}
which induce
  isomorphisms
  $\overline{\Bbbk A_\Gamma}\leftrightarrow\overline{R_\Gamma}$ by
  universal properties. Since $\gr R_\Gamma\cong R_\Gamma$, the result
  (and the last statement of Proposition~\ref{prop:Magnus}) follow.
\end{remark}

\begin{proof}[Proof of Theorem~\ref{thm:dim}]
  Lemma~\ref{lem:dim} gives an isomorphism between
  $\varpi(A_\Gamma)^n/\varpi(A_\Gamma)^{n+1}$ and the degree-$n$ part
  of $R_\Gamma$. Since $R_\Gamma$ is graded and not only filtered, its
  degree-$n$ part is   $\varpi^n/\varpi^{n+1}$, so we get the desired
  isomorphism $\varpi^n/\varpi^{n+1}\cong
  \varpi(A_\Gamma)^n/\varpi(A_\Gamma)^{n+1}$ for each $n\in\naturals$. 
\end{proof}

\subsection{The Magnus map}

We turn to the fundamental tool we use in relating the group $A_\Gamma$ with the algebra $R_\Gamma$: it is the ``Magnus map''
\begin{equation}\label{eq:Magnus_map}
  \mu\colon\left\{\begin{array}{r@{\;}l}
    A_\Gamma &\to 1+\overline\varpi\subseteq\overline{R_\Gamma}^\times\subseteq\overline{R_\Gamma},\\
    v &\mapsto 1+v\text{ for }v\in V.
  \end{array}\right.
\end{equation}
Here, $\overline{R_\Gamma}^\times$ is the group of multiplicative units of
$\overline{R_\Gamma}$. We have to map to the completion because we have to map
$v^{-1}$ to $\mu(v)^{-1}=1-v+v^2-v^3+\cdots$ which is an infinite sum. It is immediate that the commutation relations between the $v\in V$ defining $A_\Gamma$ also hold between the $\mu(v)$, therefore $\mu$ is well defined.

It is easy to describe quite explicitly a basis of the polynomial ring in partially commuting variables $R_\Gamma$. This comes
hand-in-hand with a kind of normal form for elements of $A_\Gamma$:

\begin{definition}\label{def:moves}
  A word $v_1^{e_1}\cdots v_n^{e_n}$ with $v_i\in V$ and $e^i\in\integers$ is
  called $\Gamma$-\emph{reduced} if the number $n$ of factors $v_i^{e_i}$ cannot be reduced by application of any sequence of moves which are either
  \begin{enumerate}
  \item[(M1)] remove $v_{i}^0$,
  \item[(M2)] replace the piece $v_i^{e_i}v_{i+1}^{e_{i+1}}$ by $v_i^{e_i+e_{i+1}}$ (if $v_i=v_{i+1}$), or
  \item[(M3)] replace $v_i^{e_i}v_{i+1}^{e_{i+1}}$ by $v_{i+1}^{e_{i+1}}v_i^{e_i}$ (if $\{v_i,v_{i+1}\}\in E$).
  \end{enumerate}
  Note that none of these moves increases the number of factors.
\end{definition}

\noindent We then immediately get the
\begin{lemma}
  The set of (M3)-equivalence classes of $\Gamma$-reduced words is a
  basis of $R_\Gamma$; more precisely, any set of representatives of
  (M3)-equivalence classes of reduced words of length $n$ forms a
  basis of the degree-$n$ component of $R_\Gamma$.\qed
\end{lemma}

In case $\Bbbk=\integers$, or more generally if $\Bbbk$ has characteristic
$0$, it is known that the Magnus map $\mu$ is injective,
see~\cite{Wade}*{Corollary~4.8}. We adapt this argument to $\Bbbk$ of non-zero
characteristic, arriving at some of the original results of this note:
\begin{lemma}\label{lem:wade}
  Let $\Bbbk$ be a ring of characteristic $p>0$.

  Consider $g\in A_\Gamma$. There exists a maximal $k\in\naturals$, and
  minimal $s_1,\dots,s_k\in\naturals$, such that there is a $\Gamma$-reduced monomial $m=w_1^{p^{s_1}}\cdots w_k^{p^{s_k}}$ with non-zero coefficient in $\mu(g)$. This monomial is unique. Furthermore, if $v_1^{e_1}\cdots v_n^{e_n}$ is a reduced representative of $g$ then $n=k$ and $v_1\cdots v_n=w_1\cdots w_k$ and $p^{s_i}|e_i$ and the coefficient of $m$ in $\mu(g)$ is $(e_1 p^{-s_1})\cdots(e_n p^{-s_n})$.
\end{lemma}
\begin{proof}
  Consider a $\Gamma$-reduced representative $v_1^{e_1}\cdots v_n^{e_n}$ of $g$. By definition,
  \begin{equation*}
\mu(v_1^{e_1}\cdots v_n^{e_n}) = (1+v_1)^{e_1}\cdots (1+v_n)^{e_n}
\end{equation*}
which is a possibly infinite (if one of the $e_i$ is less than $0$) $\finfield$-linear combination of words over $V$. Write $e_j=p^{s_j}\ell_j$ so that $p$ does not divide $\ell_j$. Because we are in characteristic $p$, we have $(1+v_j)^{e_j}=(1+v_j^{p^{s_j}})^{\ell_j}$.

  We may now apply a variant of Magnus's original
  argument~\cite{Magnus:1935}*{Satz~I}: multiplying out (using the power
  series for the inverse), we obtain a multiple of $v_1^{p^{s_1}}\cdots
  v_n^{p^{s_n}}$ precisely once, with coefficient $\ell_1\cdots\ell_n\ne
  0\in\finfield$. Other terms either have fewer syllables or larger
  exponents. The monomial $v_1\cdots
  v_n$ and all other monomials with the same number of syllables and
  possibly larger exponents are $\Gamma$-reduced, because any sequence of moves which
  would reduce one of them could be applied in the same way to the original
  $v_1^{e_1}\cdots v_n^{e_n}$  and would reduce its number of factors, as
  well. Therefore the term $v_1^{p^{s_1}}\cdots  v_n^{p^{s_n}}$ indeed is
  uniquely determined as the $\Gamma$-reduced monomial in $\mu(g)$ with non-zero
  coefficient with maximal number of syllables and minimal exponents.

  Since $\mu(g)$ is independent of the choice of representative of $g$, every other $\Gamma$-reduced representative $(v'_1)^{e'_1}\cdots (v'_{n'})^{e'_{n'}}$ must satisfy $n=n'$ and $v_1\cdots v_n=v'_1\cdots v'_{n'}$.  
\end{proof}

From this (and we note it for further use) we may deduce that every element of $A_\Gamma$ has an essentially unique reduced representative:
\begin{proposition}[\cite{Wade}*{Theorem~4.14}]\label{prop:normal_form}
  If $v_1^{e_1}\cdots v_n^{e^n}$ and $w_1^{f_1}\cdots w_m^{f_m}$ are two reduced words representing the same element of $A_\Gamma$, then one can be obtained from the other by a finite number of applications of~(M3). In particular, $n=m$.
\end{proposition}
\begin{proof}
  We note first by Lemma~\ref{lem:wade} that $m=n$. We then proceed by
  induction on $m$. Consider the equal elements $v_2^{e_2}\cdots v_m^{e_m}$
  and $v_1^{-e_1}w_1^{f_1}\cdots w_m^{f_m}$. The latter is not $\Gamma$-reduced, again
  by Lemma~\ref{lem:wade}, so there must exist $k\in\naturals$ with $w_k=v_1$
  and $\{v_1,w_i\}\in E$ for all $i\le k$. If $f_k\neq e_1$ then
  $w_1^{f_1}\cdots w_k^{f_k-e_1}\cdots w_m^{f_m}$ is $\Gamma$-reduced, yet again
  contradicting Lemma~\ref{lem:wade}, so $f_k=e_1$ and we apply induction to
  $v_2^{e_2}\cdots v_m^{e_m}$ and $w_1^{f_1}\cdots\widehat{w_k^{f_k}}\cdots
  w_m^{f_m}$, where the factor with hat is left out.
\end{proof}
  
\begin{proposition}\label{prop:Magnus}
  For arbitrary $\Bbbk$, the Magnus map $\mu\colon A_\Gamma\to\overline{R_\Gamma}$ is injective.
  
  It maps $\gamma_n(A_\Gamma)$ into the subgroup $1+\overline\varpi^n$ of
  $1+\overline\varpi\subset \overline{R_\Gamma
}$. We get an induced map of graded Lie algebras 
  \[\mu_L\colon \bigoplus_{n\ge 1}
    \gamma_n(A_\Gamma)/\gamma_{n+1}(A_\Gamma)\to \bigoplus_{n\ge 1}
    (1+\overline\varpi^n)/(1+\overline\varpi^{n+1})\cong \bigoplus_{n\ge 1}
    \varpi^n/\varpi^{n+1}\subset R_\Gamma,\]
  where the Lie algebra structure of $R_\Gamma$ is the one induced from the algebra structure.

  The algebra map induced by $\mu$ on the group algebra $\Bbbk A_\Gamma$ extends
  continuously to an isomorphism of filtered associative $\Bbbk$-algebras
  \begin{equation*}
  \overline{\mu}\colon \overline{\Bbbk A_\Gamma}\xrightarrow{\cong}
  \overline{R_\Gamma}.
\end{equation*}
  In particular,
  \begin{equation*}
\Bbbk A_\Gamma/\varpi^n(A_\Gamma)\cong\overline{\Bbbk
  A_\Gamma}/\overline{\varpi^n(A_\Gamma)}\cong
\overline{R_\Gamma}/\overline{\varpi^n}=R_\Gamma/\varpi^n  \cong \gr(\Bbbk
A_\Gamma)/\gr(\Bbbk A_\Gamma)_{\ge n},
\end{equation*}
using Lemma~\ref{lem:dim} for the last isomorphism. As $\Bbbk$-modules, these
are of course  also isomorphic to $(R_\Gamma)_{<n}\cong \gr(\Bbbk
A_\Gamma)_{<n}$. 
\end{proposition}
\noindent 
\begin{proof}
  Let $\Bbbk'$ be the image of $\integers$ in $\Bbbk$; it is either
  $\integers$ or $\integers/N$ for some integer $N$. The case $\integers$ is
  already covered; if $\Bbbk'=\integers/N$, let $p$ be a prime number dividing
  $N$. We prove the stronger statement that the composition
  $A_\Gamma\xrightarrow{\mu}\overline{R_\Gamma}\to\overline{R_\Gamma\otimes_\integers\finfield}$
  is injective, i.e.,~we assume without loss of generality that
  $\Bbbk=\finfield$. Injectivity of $\mu$ for $\Bbbk=\finfield$ directly
  follows from Lemma~\ref{lem:wade}.
    
  It is an elementary calculation in non-commutative power series that
  the $1+\overline\varpi^n$ form a central series of subgroups of
  $1+\overline\varpi$. By the minimality and functoriality of the lower
  central series, 
  \begin{equation*}
  \gamma_n(1+\overline\varpi)\subseteq1+\overline\varpi^n\qquad\text{and then}\qquad
  \mu(\gamma_n(A_\Gamma))\subseteq1+\overline\varpi^n.
\end{equation*}
Elementary
  calculations in the non-commutative power series ring also show
  that we have an isomorphism of associated graded Lie algebras
  \begin{equation*}
  \bigoplus_{n\ge 1} (1+\overline\varpi^n)/(1+\overline\varpi)^{n+1}
  \xrightarrow{\cong} \bigoplus_{n\ge 1} \varpi^n/\varpi^{n+1}; [1+w]\mapsto [w]
\end{equation*}
where the right hand side is the
  graded Lie algebra structure underlying the associated graded algebra 
  $R_\Gamma$ (with only the central summand $\varpi/\varpi^1$ of $R_\Gamma$
  missing). As $R_\Gamma$ is already a graded algebra, it coincides with its
  associated graded. For details of these
  computations, compare e.g.~\cite{Wade}*{Lemma 4.10}.
  
  Finally, the induced algebra map
  $\Bbbk A_\Gamma\to \overline{R_\Gamma}$ is compatible with the augmentation
  homomorphisms as the same is true for the initial map $\mu\colon A_\Gamma
  \to 1+\overline{\varpi}$ (all elements on the left and on the right have augmentation
  $1$). Consequently, it  preserves the
  filtrations by powers of the augmentation ideals and induces a homomorphism
  $\gr(\mu)$ 
  on the associated
  graded algebra. On the generating set $V$ this homomorphism
  is evidently the inverse of the map $\alpha$ of
  Lemma~\ref{lem:dim}.

  We learn that our homomorphism of complete filtered
  algebras $\overline{\mu}\colon
  \overline{\Bbbk A_\Gamma}\to \overline{R_\Gamma}$  induces an
  isomorphism of the associated graded algebras. By general theory therefore
  $\overline{\mu}$ itself is an isomorphism. In more detail,
 $\overline{\Bbbk A_\Gamma}$ is the inverse limit of the
  $\overline{\Bbbk A_\Gamma}/\overline{\varpi^n(A_\Gamma)}$, and
  correspondingly for $\overline{R_\Gamma}$. Inductively and using the
  $5$-lemma, $\overline{\mu}/\overline{\varpi^n}\colon \overline{\Bbbk
    A_\Gamma}/\overline{\varpi^n(A_\Gamma)} \to
  \overline{R_\Gamma}/\overline{\varpi^n}$ is an isomorphism (as
  $\overline{\mu}/\overline{\varpi^n}$ is the extension of
  $\overline{\mu}/\overline{\varpi^{n-1}}$ by the isomorphism
    $\gr(\mu)_n$). Finally,   $\overline{\mu}$ is an isomorphism as limit of isomorphisms.
\end{proof}

\section{Cohomology}\label{ss:coh}

A (topological) way to define and compute the cohomology of a discrete group
$G$ is via a \emph{classifying space} $X_G$. By definition, this is a connected CW-cell complex with $\pi_1(X_G)=G$ whose universal covering is
contractible. We then have $H^*(G;\Bbbk)=H^*(X_G;\Bbbk)$.

\begin{proof}[Proof of Theorem~\ref{thm:coh}, first claims]
To compute the structure of the cohomology ring $H^*(A_\Gamma;\Bbbk)$,
we first show that $X_\Gamma$ of \eqref{eq:classifying_space} is a (particularly nice) classifying space for
$A_\Gamma$. The space $X_\Gamma$ inherits a CW-cell structure (indeed a cube
complex structure) from the product
cell structure of $(S^1)^V$, where $S^1$ has just one $0$-cell $\{*\}$
consisting of the base point and one
$1$-cell. Then $X_\Gamma$ has a single vertex $*^V$ and precisely one loop
$(S^1)^{\{v\}}\times\{*\}^{V\setminus\{v\}}$ for each generator $v\in V$. The
$2$-cells in $X_\Gamma$ give the commutation relations.
By the standard computation of the fundamental group of CW-complexes (based on
the van Kampen theorem) we then have
$\pi_1(X_\Gamma,*^V)=A_\Gamma$.

Furthermore, the link of the single vertex in $X_\Gamma$ is a flag complex, since every subset of a
clique is a clique. Therefore, $X_\Gamma$ is a cube complex whose link
is a flag complex, so $X_\Gamma$ is a locally CAT(0)
space~\cite{Gromov:1987},
see~\cite{BridsonHaefliger:1999}*{Theorem~5.18}, so its universal
cover is contractible.

  The cells given in the expression of $X_\Gamma$ above form a basis
  of the homology of $X_\Gamma$: the
  differentials in the cellular chain complex vanish identically, because
  every cell sits in a subcomplex which is the cellular chain complex of a
  torus with precisely this property. Note that we get a basis of
  $H_*(X_\Gamma;\Bbbk)$ as free $\Bbbk$-module by the images of the
  fundamental classes of all subtori $T^C$ where $C$ runs through the cliques
  in $\Gamma$. As the homology is finitely generated free, the cohomology is canonically the
  dual of the homology. We see that $H^*(X_\Gamma;\Bbbk)$ is precisely the
  quotient of the exterior algebra
  $H^*(T^V;\Bbbk)= \textstyle{\bigwedge}^*(\Bbbk V)$, the cohomology of the ambient torus
  $T^V$, by the submodule generated by all products  $v_1\dots v_r$ such that
  $v_1,\dots,v_r$ do not span a clique in $\Gamma$. The comparison map is
  induced by the inclusion $X_\Gamma\hookrightarrow T^V$. That this map is surjective
  with the claimed kernel follows by naturality and the know (co)homology of
  $T^V$, together with the information about the rank of $H^*(X_\Gamma;\Bbbk)$
  we obtained from the cellular complex. Now the quotient algebra is precisely the
  algebra $S_\Gamma$ and we have proven
  $H^*(A_\Gamma;\Bbbk)=H^*(X_\Gamma;\Bbbk)=S_\Gamma$ as algebras.
\end{proof}

We note that $H^*(X_\Gamma;\Bbbk)=S_\Gamma$ has a natural $\Bbbk$-basis indexed by cliques
$C$ in $\Gamma$: a degree-$k$ basis element corresponding to a clique
$C=\{v_0,\dots,v_{k-1}\}$ is given by the product
$v_C\coloneqq v_{k-1}\cdots v_0$ ---to make this definite, we pick a
total ordering of the vertices and write the factors in decreasing
order.

\subsection{Koszul algebras}

Back to general theory, consider a graded associative algebra $R$
presented as $T(W)/I$ for a finitely generated free $\Bbbk$-module $W$, its
tensor algebra $T(W)$  and an ideal
$I\le T(W)$. In case $I$ is generated by a subspace $I_2$ of
$W^{\otimes2}$, the algebra is called \emph{quadratic}; and it then
admits a \emph{quadratic} dual $R^!\coloneqq T(W^*)/(I_2^\perp)$; here
by $I_2^\perp$ we mean the subset of
$(W^*)^{\otimes2}\cong(W^{\otimes2})^*$ annihilating $I_2$. Clearly
$R^{!!}\cong R$. Now, with $\Bbbk V$ the free $\Bbbk$-module with
basis $V$, setting
\begin{equation*}
  \begin{split}
    G_R &:=\langle v\otimes w-w\otimes v\text{ for }\{v,w\}\in E\rangle \subset
    \Bbbk V^{\otimes 2},\\
    G_S&:= \langle v\otimes w\text{ for
             }\{v,w\}\notin E, v\otimes w+w\otimes v\text{ for }\{v,w\}\in
             E\rangle \subset \Bbbk V^{\otimes 2},
  \end{split}
\end{equation*}
we have as algebras
\begin{equation*}
  R_\Gamma = T(\Bbbk V)/\langle G_R\rangle\quad\text{and}\
  S_\Gamma= {\textstyle\bigwedge}^*(\Bbbk V)/\langle v\wedge w\text{ for
             }\{v,w\}\notin E\rangle     = T(\Bbbk V)/\langle G_S\rangle.     
\end{equation*}
Let us identify $\Bbbk V^{\otimes 2}$ with $(\Bbbk V^{\otimes 2})^*$
via the basis $\{v\otimes w\mid v,w\in V\}$ and its dual basis. Then
$G_S$ is the annihilator of $G_R$ (they clearly annihilate each other, and the
ranks add up to the total dimension $|V|^2$), and therefore $R_\Gamma$ and $S_\Gamma$ are
quadratic duals of each other.

Returning to generality, recall that a quadratic algebra $R$ is called
\emph{Koszul} if its \emph{Koszul complex} is acyclic,
\cite[3.4.7]{LodayVallette}. We recall the Koszul complex (in our concrete
situation) below and we mention that this is only one of a number of different
equivalent characterizations of the Koszul property. It implies that the Yoneda algebra
$\Ext_R(\Bbbk,\Bbbk)$ is isomorphic to $R^!$, compare \cite[Theorem
2.5]{Priddy:1970}.

\begin{proof}[Proof of Theorem~\ref{thm:coh}, second claim]
We now show that $R_\Gamma$ and $S_\Gamma$ are Koszul. Deliberately, we are a
bit brief as we believe that this is mainly of interest to readers which have the
required background.  In fact, a quadratic algebra is Koszul if and only if
its quadratic dual is \cite[Proposition 3.4.8]{LodayVallette}. Therefore it
suffices to prove the Koszul property for $S_\Gamma$, and there is a simple
sufficient (but not necessary) condition, the existence of a quadratic
Gr\"obner basis. Recall that a \emph{Gr\"obner basis} for an ideal
$I\le\bigwedge^*(\Bbbk V)$ is a set $G$ of generators for $I$ such
that the leading terms (with respect to a compatible order of
monomials) of elements of $G$ generate the same ideal as the leading
terms of all elements of $I$. Now
$G\coloneqq\{v\wedge w\mid\{v,w\}\notin E\}$ is a Gr\"obner basis, as
follows from \emph{Buchberger's criterion}: ``for all $f,f'\in G$
whose respective leading terms $g,g'$ have least common multiple
$\ell$, the syzygy $(\ell/g)f-(\ell/g')f'$ must vanish''.

Alternatively and without using Gr\"obner basis, the work of Fr\"oberg~\cite{Froberg:1975}*{in particular Section~3}
also implies that $R_\Gamma$ (and $S_\Gamma$) are Koszul. His proof runs
essentially as follows and uses directly the \emph{Koszul complex} of
$R_\Gamma$ which we now construct. Consider the right
$R_\Gamma$-module $P_*=\Hom_\Bbbk(S_\Gamma,R_\Gamma)$. Recall that,
qua $\Bbbk$-module, $S_\Gamma$ is finitely generated free with basis
indexed by cliques in $\Gamma$. Consequently, this basis induces and
isomorphism $P_*\xrightarrow{\cong} \bigoplus_{C} v_C R_\Gamma$, where the sum is over the
cliques in $\Gamma$. It is bigraded by $S_\Gamma$- and
$R_\Gamma$-degree.
Consider the map $d\colon P_*\to P_*$ with
\begin{equation*}
d(f)(p)=\sum_{v\in V}v f(v p)\quad\text{ for }f\in P_*,\; p\in S_\Gamma.
\end{equation*}
In our basis,
$d((v_{k-1}\cdots v_0) \cdot r) = \sum (-1)^j (v_{k-1}\cdots
\widehat{v_j}\cdots v_0)\cdot v_j r$.  A direct computation shows that
$d^2=0$. Note that $d$ increases the $R_\Gamma$-degree by $1$, and
decreases the $S_\Gamma$-degree by $1$, so $(P_*,d)$ becomes a chain
complex of finitely generated free $R_\Gamma$-modules, graded by
$S_\Gamma$-degree. 

To prove acyclicity of the Koszul complex $(P_*,d)$  we define a chain
contraction map $s\colon P_*\to P_{*+1}$ 
of $\Bbbk$-modules as follows. Recall that we have a $\Bbbk$-basis of $P_*$
given by elements $v_C\cdot w$ for a clique $C$ of $\Gamma$ and a
basis element $w$ of $R_\Gamma$ given as a $\Gamma$-reduced monomial over $V$ according
to Definition \ref{def:moves}. To define $s(v_C\cdot w)$ we consider two cases.
If we can write  $w=v w'$ in reduced form with $v\in V$ and with $w'$ a word in letters from $V$ in such a
manner that $v<\min C$ (for the total ordering on $V$ picked above)
and such that $C\cup\{v\}$ is a clique of $\Gamma$, then we choose $v$ minimal with this property, and we set
$s(v_C\cdot v w')\coloneqq v_{C\cup\{v\}}\cdot w'$. Otherwise, we
set $s(v_C\cdot w)\coloneqq0$.

We now carry out the elementary calculation to see that $s$ is a chain
contraction, meaning $s d+d s=1-\epsilon$, where
$\epsilon\colon P_*\to \Bbbk$ is the augmentation map, projecting onto
the summand of bidegree $(0,0)$. For this, consider $x=v_C\cdot w$. The
calculation splits into three cases.
\begin{enumerate}
\item If $C=\emptyset$ and $w=1$, then $(s d +d s)(x)=0=(1-\epsilon)(x)$. 
\item Assume that $C=\{v_0,\dots,v_k\}\neq\emptyset$ and $w$ cannot be
  written in the form $w=v w'$ as above. Then
  \begin{equation*}
d s(x)=0\quad\text{ while }\quad s d(x)= \sum (-1)^j s((v_{k-1}\cdots \widehat{v_j}\cdots v_0)\cdot v_j
  w).
\end{equation*}
 By hypothesis, no letter in $w$ can be swapped with $v_j$ and added to
  $C\setminus\{v_j\}$, so all summands vanish except the $0$th which is
  $x$. 
\item Assume that $C=\{v_0,\dots,v_k\}$ and $w$ can be written in the
  form $v_{-1}w'$ such that $C\cup\{v_{-1}\}$ is a clique in $\Gamma$, with
  $v_{-1}<\min C$, chosen minimal among all such possibilities. Then $v_{-1}$
  commutes with all $v_j$, so
  \begin{equation*}
    \begin{split}
      s d(x) &= \sum (-1)^j s(v_{C\setminus\{v_j\}}\cdot v_j v_{-1} w')=\sum_{j=0}^{k-1} (-1)^j v_{C\setminus\{v_j\}\cup\{v_{-1}\}}\cdot v_j w',\\
      d s(x) &= d(v_{C\cup\{v\}}\cdot w')=\sum_{j=-1}^{k-1} (-1)^{j+1}
      v_{C\setminus\{v_j\}\cup\{v_{-1}\}}\cdot v_j w',
    \end{split}
  \end{equation*}
  and the terms cancel pairwise except the one with $j=-1$, giving again
  $(ds+sd)(x)=x$.
\end{enumerate}
It follows that $P_*$ is a free $R_\Gamma$-resolution
of $\Bbbk$.
%
\end{proof}

We note that the usual definition of Koszul algebras is given over fields of
characteristic $0$; however, in our case, we need not impose any restriction
on the commutative ring $\Bbbk$ (other than interpreting $(\Bbbk V)^*$ as
naturally isomorphic to $\Bbbk V$), since the rings $R_\Gamma$ and $S_\Gamma$ are $\Bbbk$-free. 

\section{Central series}\label{ss:central}

\subsection{Labute's general theory}

Labute gave in~\cite{Labute:1985} a condition under which a 
presentation $\langle V\mid\mathcal R\rangle$ of a group $G$
determines a presentation of the associated Lie algebra
$L(G)\coloneqq\bigoplus_{n\ge1}\gamma_n(G)/\gamma_{n+1}(G)$. Such a
group presentation is now called ``mild'', and Anick gave
in~\cite{Anick:1987} a valuable criterion for this to happen: view all
$r\in\mathcal R$ as elements of the free associative algebra
$T(\integers V)$, under the Magnus embedding $F_V\to T(\integers
V)$. Let $n$ be such that $r-1\in\varpi^n\setminus\varpi^{n+1}$, and
let $r'$ denote the image of $r$ in the quotient
$\varpi^n/\varpi^{n+1}$. Then $\langle V\mid\mathcal R\rangle$ is mild
if and only if $\{r'\mid r\in\mathcal R\rangle$ is ``inert''. We need not define here the meaning of ``inert'' (a.k.a.\ ``strongly free'', see~e.g.~\cite{HalperinLemaire:1987}), but merely note that there
are powerful sufficient conditions guaranteeing that a set is inert in
the free associative algebra, one of them being that it forms a
Gr\"obner basis.  It follows then quite generally that the Lie algebra
$L(G)$ admits as presentation
$\langle V\mid r'\;\forall r\in\mathcal R\rangle$,
see~\cite{Labute:1985}*{Theorem~1}; and a similar statement holds for
the restricted Lie algebra
$\bigoplus_{n\ge1}\lambda_{n,p}(G)/\lambda_{n+1,p}(G)$,
see~\cite{Labute:1985}*{Theorem~3}. Labute's conditions are non-trivial to check, so we shall in fact recover his results rather than
use them.

\subsection{First easy results for RAAGs}

By Proposition~\ref{prop:Magnus} the rings $R_\Gamma/\varpi^n$ and $\Bbbk A_\Gamma/\varpi(G)^n$ are isomorphic, so the dimension subgroups $\delta_{n,\mu}$ and $\delta_{n,\Bbbk A_\Gamma}$ are equal. Furthermore, since the Magnus map $\mu$ has image in the subring of $R_\Gamma$ generated by $1$ and $V$, the groups $\delta_{n,\mu}$ depend on $\Bbbk$ only via the image $\Bbbk'$ of $\integers$ in $\Bbbk$.

We consider two cases: if $\integers\subseteq\Bbbk$ then the dimension subgroups associated with the rings $\Bbbk$ and $\rationals$ agree. If, on the other hand, $\finfield\subseteq\Bbbk$, then the dimension subgroups associated with the rings $\Bbbk$ and $\finfield$ agree. In all cases, we reduce to the case $\Bbbk\in\{\finfield,\integers\}$.

\begin{proof}[Proof of Theorem~\ref{thm:cs1}]
  We  apply the classical results of Jennings and Hall. For $\Bbbk=\rationals$ we have $\gamma_{n,0}=\delta_{n,\Bbbk A_\Gamma}$; compare~\cites{Jennings:1955,Hall} which treat the case of torsion-free nilpotent groups to which the general case easily reduces. For $\Bbbk=\finfield$ we have $\gamma_{n,p}=\delta_{n,\Bbbk A_\Gamma}$; compare~\cite{Jennings:1941} which treats the case of finite $p$-groups to which the general case easily reduces.
\end{proof}

\section{Lie algebras associated with $\Gamma$}\label{ss:lie}
Recall that the cohomology of a Lie algebra $L$, defined as $\Ext_{U(L)}(\Bbbk,\Bbbk)$, may be computed using
its \emph{Chevalley complex} $({\bigwedge}^*(L^\circ),d)$, with
$L^\circ$ the ``small dual'' of $L$, namely
\begin{equation*}
L^\circ=\{\phi\in L^*\mid \ker\phi\text{ contains a
  finite-codimensional ideal}\},
\end{equation*}
and the differential
$d\colon L^\circ\to\bigwedge^2 L^\circ$ is the dual of the Lie
bracket map $\bigwedge^2L\to L$ (extended to all degrees by requiring $d$ to
be a graded derivation). Note that $L^\circ$ is just so
defined that the image of $d$ belongs to
$\bigwedge^2 L^\circ\subset (\bigwedge^2 L)^*$. Since
$\bigwedge^*(L^\circ)$ is a graded commutative algebra and $d$ is a
derivation, the homology $(\bigwedge^*(L^\circ),d)$ is naturally a graded
commutative algebra.

\begin{proof}[Proof of Theorem~\ref{thm:lie1}, Lie algebra cohomology of $L_\Gamma$]
  The enveloping algebra of $L_\Gamma$ is $R_\Gamma$, which is Koszul
  with Koszul dual $S_\Gamma$, so we have
  $$ H^*(L_\Gamma;\Bbbk)=H^*(\bigwedge^*(L^\circ),d)=\Ext_{R_\Gamma}(\Bbbk,\Bbbk)=S_\Gamma.$$
Note that $\bigwedge^* L_\Gamma^\circ$ admits two gradings, one as an
exterior algebra and one inherited from the grading of $L_\Gamma$. In
$H^*(\bigwedge^*(L^\circ),d)$, these two gradings coincide --- this is
precisely the content of $S_\Gamma$ being a Koszul algebra.
\end{proof}

In the following, we write $L$ for $L_{\Gamma}$ if the characteristic of $\Bbbk$
is $0$, and for
$L_{\Gamma,p}$ if the characteristic of $\Bbbk$ is $p$, and view $L$ as a subset of $R_\Gamma=U(L)$.
Following Magnus' method~\cite{MagnusKarrassSolitar}*{Theorem~5.12}, consider
$x\in L_n$, i.e.~homogeneous of degree $n$. Then $x$ is a linear combination
(with coefficients in $\Bbbk$) of a collection of bracket arrangements
$\phi_i=\phi_i(v_1,\dots,v_n)$. The assignment
\begin{equation*}
L_n\ni\phi_i
\mapsto\phi_i(v_1,\dots,v_n)\in \gamma_n\subseteq A_\Gamma
\end{equation*}
is well defined
on the subset of bracket arrangements, since $[v,w]=1\in A_\Gamma$ for each
$\{v,w\}\in E$. It extends $\Bbbk$-linearly to a map 
\[\nu\colon L_n\to \gamma_n/\gamma_{n+1}\otimes_\integers\Bbbk\]
of $\Bbbk$-modules. This map is clearly surjective, since
$\gamma_n/\gamma_{n+1}$ is spanned by $n$-fold bracket arrangements, for an arbitrary
group. Furthermore, the composition $\mu_L\circ\nu\colon L\to R_\Gamma$ with
$\mu_L$ given in Proposition~\ref{prop:Magnus} is a Lie algebra map sending $v$ to
$v$. Therefore this composition is the inclusion of $L$ into $R_\Gamma$ and is
in particular injective. This implies that $\nu$ is an isomorphism with
inverse the Magnus map $\mu_L$.

\begin{proof}[Proof of Theorem~\ref{thm:cs2}, characteristic $0$]
  Consider $\Bbbk=\integers$. Since $L_\Gamma$ is $\integers$-free, it follows in particular that $\gamma_n(A_\Gamma)/\gamma_{n+1}(A_\Gamma)$ is torsion-free for each $n$, and therefore $\gamma_{n,0}(A_\Gamma)=\gamma_n(A_\Gamma)$ for all $n$.
\end{proof}

\begin{proof}[Proof of Theorem~\ref{thm:lie2}, first two claims]\strut\\
  The isomorphism $\nu$ identifies $L_\Gamma$ and $\bigoplus_{n\ge1}(\gamma_n/\gamma_{n+1})\otimes_\integers\Bbbk$.
\end{proof}

\begin{lemma}\label{lem:assgraded_exponentp}
  Consider $\Bbbk=\integers$, and define the ideal
  $\varpi_p=\langle p,V\rangle$ of $R_\Gamma$.

  The associated graded ring
  $\bigoplus_{n\ge0}\varpi_p^n/\varpi_p^{n+1}$ is isomorphic to
  $R_\Gamma\otimes_\integers\finfield[\pi]$, with $\pi$ of degree $1$
  mapped to $[p]\in \varpi_p/\varpi^2_p$ under the isomorphism.
\end{lemma}
\begin{proof}
  Powers of $\varpi_p$ define a new filtration on $R_\Gamma$, in which $v\in V$ still has degree $1$, but in addition $p$ also has degree $1$; thus for instance $p^2v^3$ belongs to the fifth term of the filtration. The ring $R_\Gamma$ is $\integers$-free. When passing to the associated graded ring for the new grading, we get on the one hand $\bigoplus \varpi_p^n/\varpi_{p}^{n+1}$. On the other hand, this graded ring is obtained from the old associated graded (which is the graded algebra $R_\Gamma$) by replacing each copy of $\integers$ by its own associated graded under the filtration $(p^n)$, namely by $\finfield[\pi]$. This replacement amounts to tensoring over $\integers$ with $\finfield[\pi]$.
\end{proof}

In case $p\ge 3$, we are now ready to identify the non-restricted Lie
algebra $\bigoplus_{n\ge1}\lambda_{n,p}/\lambda_{n+1,p}$ with
$L_\Gamma\otimes_\integers\finfield[\pi]$. Let us temporarily write
$\beta_n\coloneqq\mu^{-1}(1+\overline\varpi_p^n)$. We make the following claim.
\begin{lemma}\label{lem:exponentp}
  For $p\ge 3$ prime, the Magnus map $\mu$ induces a composition of
  (non-restricted) Lie algebra isomorphisms over $\finfield[\pi]$, still written $\mu_L$,
  \[\mu_L\colon \bigoplus_{n\ge1}\lambda_{n,p}/\lambda_{n+1,p}\to\bigoplus_{n\ge1}\beta_n/\beta_{n+1}\to L_\Gamma\otimes_\integers\finfield[\pi],\]
  with the first map induced by inclusion $\lambda_{n,p}\le\beta_n$ and the second map induced by $\beta_n/\beta_{n+1}\ni [1+a]\mapsto a\in\varpi_p^n/\varpi_p^{n+1}$.
  
  In particular, we have $\beta_n=\lambda_{n,p}$.
\end{lemma}
\begin{proof}
  To check that the first map is well-defined, it suffices to show
  $\lambda_{n,p}\le\beta_n$. We have
  $\varpi_p^n=\sum_{m+i\ge n}p^i\varpi^m$. Consider $g\in\gamma_m$, so
  by definition $\mu(g)=1+x$ for some $x\in\varpi^m$. We then have
  $\mu(g^{p^i})=(1+x)^{p^i}=1+p^i x+\cdots\in1+\varpi_p^{m+i}$, so
  $\mu(\gamma_m^{p^i})\subseteq 1+\varpi_p^{m+i}$. Since
  $\lambda_{n,p}=\prod_{m+i\ge n}\gamma_m^{p^i}$, we have shown
  $\lambda_{n,p}\le\beta_n$.

  Because the Magnus map
  $\mu\colon A_\Gamma\to 1+\overline\varpi\subset R_\Gamma$ is
  injective by Proposition~\ref{prop:Magnus}, so is the induced map
  $\beta_n/\beta_{n+1}\to
  (1+\overline\varpi_p^n)/(1+\overline\varpi_p^{n+1})=\varpi_p^n/\varpi_p^{n+1}$,
  which is our second map.
  
  Since $p\ge 3$, the assignment $\pi\cdot [g]\coloneqq [g^p]$ for
  $g\in\lambda_{n,p}$ (with $g^p\in\lambda_{n+1,p}$) gives
  $\bigoplus \lambda_{n,p}/\lambda_{n+1,p}$ the structure of an
  $\finfield[\pi]$-module. For this we use the Hall-Petrescu
  identities~\cite{Hall:1934}*{Theorems~3.1, 3.2}: if $g,h$ belong to
  an arbitrary group $G$, then
  $(gh)^p \equiv g^p h^p [h,g]^n\alpha(g,h)$ with $n=\binom p2$ and
  $\alpha(\cdot,\cdot)$ a universal expression in $\gamma_3$. This implies
  $(gh)^p=g^p h^p\mod{\gamma_{n+2,p}}$ for $g,h\in\gamma_{n,p}$
  if either $n\ge 2$ or $n=1$ and $p\ge 3$. However, beware that if
  $n=1$ and $p=2$ then this does not hold in general, so the $p$-power
  operation is not linear. We see that $\mu_L$ maps this $p$-power
  operation to multiplication by $\pi$ on
  $L_\Gamma\otimes_\integers\finfield[\pi]$. It follows that $\mu_L$
  is an $\finfield[\pi]$-Lie algebra homomorphism. Its image contains
  $V$ which generates $L_\Gamma$, so $\mu_L$ is surjective. Finally,
  $L_\Gamma\otimes_\integers\finfield[\pi]$ is the free Lie algebra
  over $\finfield[\pi]$ modulo the relations $[v,w]=0$ for
  $\{v,w\}\in E$. Those relations are clearly satisfied in the
  $\finfield[\pi]$-Lie algebra
  $\bigoplus_{n\ge1}\lambda_{n,p}/\lambda_{n+1,p}$, so the map $\mu_L$
  is an isomorphism.
  
  It then follows that the second map is surjective and therefore an isomorphism, so the first is also bijective, from which we deduce $\beta_n=\lambda_{n,p}$.
\end{proof}

\begin{proof}[Proof of Theorem~\ref{thm:cs2}, characteristic $p$]
  Let $\Bbbk$ be an algebra over $\finfield$. By~\cite{Quillen:1968},
  the Lie algebra
  $\bigoplus_{n\ge1}(\gamma_{n,p}/\gamma_{n+1,p})\otimes_\integers\Bbbk$
  is isomorphic to the primitive subalgebra of
  $\bigoplus_{n\ge0}\varpi(A_\Gamma)^n/\varpi(A_\Gamma)^{n+1}\cong
  R_\Gamma$, namely to $L_{\Gamma,p}$.
\end{proof}

\begin{proof}[Proof of Theorem~\ref{thm:lie2}, last claim]
  This is precisely Lemma~\ref{lem:exponentp}.
\end{proof}

\section{Growth series}\label{ss:growth}
We derive now some relations between the Poincar\'e series of $S_\Gamma$, $R_\Gamma$, $L_\Gamma$ and $L_{\Gamma,p}$ from general considerations. We recall that, for a graded algebra $R=\bigoplus_{n\ge0}R_n$, its Poincar\'e series is $\Phi_R(t)=\sum_{n\ge0}\rank(R_n)t^n$.

\begin{proof}[Proof of Theorem~\ref{thm:growth}]
First, we use Koszul duality between $R_\Gamma$ and $S_\Gamma$ to deduce
$\Phi_{R_\Gamma}(t)\cdot\Phi_{S_\Gamma}(-t)=1$, compare \cite[Theorem 3.5.1]{LodayVallette}.
This relationship between the Poincar\'e series of $R_\Gamma$ and $S_\Gamma$ was already noted in~\cites{CartierFoata:1969,SheltonYuzvinsky:1997}.

We have
$\Phi_{S_\Gamma}(t)=\sum_{n\ge0}\rank
H^n(A_\Gamma,\Bbbk)t^n=\sum_{n\ge0}c_n(\Gamma)t^n$, with $c_n(\Gamma)$
the number of $n$-cliques in $\Gamma$, from our explicit basis of
$S_\Gamma$ given in Section~\ref{ss:coh}.

The relation between $\Phi_{R_\Gamma}$ and $\Phi_{L_\Gamma}$ is given by the Poincar\'e-Birkhoff-Witt theorem, namely the fact that $R_\Gamma$ and the symmetric algebra over $L_\Gamma$, respectively the degree-$p$ truncated symmetric algebra over $L_{\Gamma,p}$, are isomorphic as graded $\Bbbk$-modules. It is expressed by the relation
\[\sum_{n\ge0}a_n t^n=\prod_{n\ge1}\left(\frac1{1-t^n}\right)^{b_n}=\prod_{n\ge1}\left(\frac{1-t^{p n}}{1-t^n}\right)^{c_n}\]
if $\Phi_{R_\Gamma}(t)=\sum_{n\ge0}a_n t^n$, $\Phi_{L_\Gamma}(t)=\sum_{n\ge1}b_n t^n$, and $\Phi_{L_{\Gamma,p}}(t)=\sum_{n\ge1}c_n t^n$.

Finally, we consider the growth series of the group $A_\Gamma$. It is the function $\Phi_{A_\Gamma}(t)=\sum_{g\in A_\Gamma}t^{\|g\|}$, with $\|g\|$ the minimal number of terms of $V\cup V^{-1}$ required to write $g$ as a product. We cite~\cite{AthreyaPrasad:2014}:
\[\Phi_{A_\Gamma}(t) = \Phi_{R_\Gamma}\bigg(\frac{2t}{1+t}\bigg).\]
Indeed, as we saw in Proposition~\ref{prop:normal_form}, every element
$g\in A_\Gamma$ can be written in the form
$g=v_1^{e_1}\cdots v_n^{e_n}$ for some $e_i\in\integers\setminus\{0\}$
as a word of minimal length; and this expression is unique up to
permuting some terms according to rule~(M3). Let $I$ be the set of (M3)-equivalence
classes $(v_1,\dots,v_n)$ of minimal-length sequences. For an element
$[v_1,\dots,v_n]$ of $I$,
the collection of
all such terms $v_1^{e_1}\cdots v_n^{e_n}$ contributes $(t+t^2+t^3+\cdots)^n =
(t/(1-t))^n$ to the 
growth series of $R_\Gamma$ because each $e_i$ can be an arbitrary positive
natural number; and it contributes $(2t/(1-t))^n$ to the growth of
$A_\Gamma$, taking into account the signs of the $e_i$. Since we
obtain all elements of $A_\Gamma$ and all basis elements of $R_\Gamma$
that way, we have
\begin{equation*}
  \Phi_{A_\Gamma}(t) = \sum_I \left(\frac{2t}{1-t}\right)^n = \sum_I
  \left(\frac{2t/(1+t)}{1-2t/(1+t)}\right)^n = \Phi_{R_\Gamma}\left(\frac{2t}{1+t}\right),
\end{equation*}
using $2t/(1-t)=(2t/(1+t))/(1-2t/(1+t))$y. We have finished the proof of
Theorem~\ref{thm:growth}.
\end{proof}

\section{Malcev completions}\label{ss:malcev}
In this section we fix $\Bbbk=\rationals$. Recall
from~\cite{PapadimaSuciu:2004} that a \emph{Malcev Lie algebra} is a
Lie algebra $L$ over $\rationals$, given with a descending filtration
$(L_n)_{n\ge1}$ of ideals such that $L$ is complete with respect to
the associated topology, and satisfying $L_1=L$ and
$[L_m,L_n]\subseteq L_{m+n}$ and such that
$\bigoplus_{n\ge1}L_n/L_{n+1}$ is generated in degree $1$. Every
Malcev Lie algebra admits an associated \emph{exponential group}
$\exp(L)$, which is $L$ as a set, with product given by the
Baker-Campbell-Hausdorff formula $x\cdot y=x+y+[x,y]/2+\cdots$.

Lazard proved in~\cite{Lazard:1954} that every group homomorphism
$\rho\colon G\to\exp(L)$ induces a morphism of graded Lie algebras
$\bigoplus_{n\ge1}\gamma_n/\gamma_{n+1}\otimes\rationals\to\bigoplus_{n\ge1}L_n/L_{n+1}$.

A \emph{Malcev completion} of a group $G$ is a homomorphism
$\rho\colon G\to\exp(L)$ for a Malcev Lie algebra $L$, universal in
the sense that every representation $G/\gamma_n\to\exp(L')$ for a
(nilpotent) Malcev Lie algebra $L'$ factors uniquely through
$\exp(L/L_n)$; see~\cite{PapadimaSuciu:2004}*{Definition~2.3}.

Quillen gave a direct construction of the Malcev completion of a group
in~\cites{Quillen:1968,Quillen:1969}: let $\overline{\rationals
  G}=\projlim\rationals G/\varpi^n$ be the completion of the group ring; then
$\overline{\rationals G}$ is a complete Hopf algebra. Let $L$ be its Lie
subalgebra of primitive elements; it is a Malcev Lie algebra for the
filtration $L_n=L\cap\overline\varpi^n$. Let $\exp\colon L\to
\overline{\rationals G}$ be the usual power series map
$\exp(x)=1+x+x^2/2+\cdots$ which makes sense in $\overline{\rationals
  G}$. Then its image $\overline G\coloneqq\exp(L)$ is a subgroup of the group of
multiplicative units. It identifies with the Lie group associated to the
Malcev Lie algebra $L$, and it consists precisely of the grouplike elements in
$\overline{\rationals G}$, namely the $g\in1+\overline\varpi$ satisfying
$\Delta(g)=g\otimes g$. The representation $\rho\colon G\to\exp(L); g\mapsto
g$ is the Malcev completion of $G$.

The Magnus map $\mu\colon A_\Gamma\to\overline{R_\Gamma}^\times$
yields an isomorphism of associative algebras
$\overline{\rationals A_\Gamma}\cong\overline{R_\Gamma}$. Both
algebras are actually complete Hopf algebras, but the Magnus
isomorphism does not preserve the Hopf algebra structure: $v\in V\subset\overline{\rationals A_\Gamma}$ is group-like, meaning $\Delta(v)=v\otimes v$ while $v\in V\subset\overline{R_\Gamma}$ is primitive, meaning $\Delta(v)=v\otimes 1+1\otimes v$; so $\Delta(\mu(v))=\Delta(1+v)=1\otimes 1+v\otimes 1+1\otimes v$ while $(\mu\otimes\mu)(\Delta(v))=(1+v)\otimes(1+v)$.

The Magnus map $\mu$ is, in fact, the truncation to order $1$ of a Hopf algebra isomorphism $\mu_{\exp}\colon\overline{\rationals A_\Gamma}\to\overline{R_\Gamma}$, given on $v\in V$ by the classical exponential series
\[\mu_{\exp}(v)=\sum_{n\ge0}\frac{v^n}{n!}=1+v+\mathcal O(v^2).
\]

\begin{proof}[Proof of Theorem~\ref{thm:malcev}]
The proof that $\mu_{\exp}$ is an isomorphism of filtered associative
algebras is exactly the same as that of Theorem~\ref{thm:cs2}, and will
not be repeated. On the other hand, the fact that $\mu_{\exp}$ is a
coalgebra map follows formally from the fact that the power series
$\exp$ maps primitive elements to group-like elements:
\begin{align*}
  \Delta(\mu_{\exp}(v)) &= \Delta\Big(\sum_{n\ge0}v^n/n!\Big) = \sum_{n\ge0}\Delta(v)^n/n!\\
                 &=\sum_{n\ge0}\frac{(v\otimes1+1\otimes v)^n}{n!}=\sum_{\ell,m\ge0}\frac{(v\otimes1)^\ell(1\otimes v)^m}{\ell!m!}\\
                 &=(\exp v\otimes1)(1\otimes\exp v)=(\mu_{\exp}\otimes\mu_{\exp})(\Delta(v)).
\end{align*}
We have proven the first claim.

It now suffices to use this isomorphism $\mu_{\exp}$ to make even more
concrete the construction of Quillen sketched above: in
$\overline{\rationals A_\Gamma}$ the space of primitive elements is
slightly mysterious, for example, it contains
\begin{equation*}
\log(g)=\log(1-(1-g))=-\sum_{n\ge1}(1-g)^n/n\quad\forall g\in A_\Gamma.
\end{equation*}
In contrast to this, its exponential is the Malcev completion
naturally containing $A_\Gamma$. In $\overline{R_\Gamma}$ the
opposite holds: the space of primitive elements is the Lie subalgebra
$L_\Gamma$ while its exponential cannot be better defined than as the
exponential of $L_\Gamma$.

In all cases, the Hopf algebra isomorphism $\mu_{\exp}$ directly
yields the remaining claims of Theorem~\ref{thm:malcev}.
\end{proof}

We now turn to formality in the sense of Sullivan in rational homotopy
theory. A finite CW-complex
$X$ is called \emph{formal} if its algebraic minimal model is
quasi-isomorphic to $(H^*(X;\rationals),0)$. This implies that the
rational homotopy type of $X$ is determined in a precise way by its
rational cohomology ring. For details on rational homotopy theory
compare~\cite{Sullivan:1977} or the more recent~\cite{Felix:2001}.

We finally prove that the space $X_\Gamma$ constructed in
Section~\ref{ss:coh} is formal.  Recall that we defined $X_\Gamma$ as
a (cubical) subspace of the smooth manifold $(\reals/\integers)^V$. It
makes perfect sense to restrict smooth differential forms on
$(\reals/\integers)^V$ to $X_\Gamma$. We define
$A^*(X_\Gamma)$ to be the algebra of all such restrictions; it is a commutative differential graded algebra (cdga). It is an easy
exercise that this cdga is quasi-isomorphic to the standard cdga over $\reals$ of
rational homotopy theory associated to $X_\Gamma$. There are basic
one-forms $dx_v$ on $(\reals/\integers)^V$ coming from the obvious
coordinate functions, for $v\in V$. Their images in $A^*(X_\Gamma)$
generate a sub-cdga with trivial differential, whose homology is
$H^*(X_\Gamma;\reals)$ by Theorem~\ref{thm:coh}. The inclusion of this
sub-cdga in $A^*(X_\Gamma)$ is a quasi-isomorphism, showing that
$X_\Gamma$ is formal.

We now explicitly exhibit a minimal model for $X_\Gamma$. Recall
from Section~\ref{ss:lie} the Chevalley complex
$(\bigwedge^*(L_\Gamma^\circ),d)$ of $L_\Gamma$. Note that $L_\Gamma$ is
graded, and $L_\Gamma^\circ$ may be 
identified with the graded dual of $L_\Gamma$. Consequently, there is a natural
map $L_\Gamma^\circ\to \rationals V$ given by restricting to the
degree-$1$ part. This map induces a map of graded algebras
$\bigwedge^*(L_\Gamma^\circ)\to \bigwedge^*(\rationals V)/\langle v\wedge
w\text{ for }\{v,w\}\notin E\rangle=S_\Gamma$. Even better, this is a map of
cdgas from the Chevalley complex to $S_\Gamma$, the latter equipped with
zero differential, and indeed is a quasi-isomorphism. These are manifestations
of the Koszul duality of $S_\Gamma$ and $R_\Gamma=U(L_\Gamma)$. As $X_\Gamma$ is formal
and $S_\Gamma=H^*(X_\Gamma;\rationals)$ we conclude that
$\bigwedge^*(L_\Gamma^\circ,d)$ is a minimal model of $X_\Gamma$.

Here is yet an alternative proof: a group is called \emph{$1$-formal}
if its Malcev Lie algebra is quadratic. It therefore follows from
Theorem~\ref{thm:malcev} that $A_\Gamma$ is $1$-formal. The cohomology
ring $H^*(X_\Gamma;\rationals)\cong S_\Gamma$ is Koszul by
Theorem~\ref{thm:coh}, so $X_\Gamma$ is formal
by~\cite{PapadimaSuciu:2006}*{Proposition~2.1}.

\section{Outlook}

\subsection{Subgroup growth}
Baik, Petri, and Raimbault determined the subgroup growth of $A_\Gamma$ in terms of the graph $\Gamma$. Define $s_n(A_\Gamma)$ as the number of subgroups of $A_\Gamma$ of index precisely $n$. Then~\cite{BaikPetriRaimbault:2018}*{Theorem A} establishes 
\begin{equation*}
    \lim_{n\to\infty} \frac{\log(s_n(A_\Gamma))}{n\log(n)} =\alpha(\Gamma)-1,
\end{equation*}
i.e.~$s_n(A_\Gamma)$ grows like $(n!)^{\alpha(\Gamma)-1}$. Here, $\alpha(\Gamma)$ is the \emph{independence number} of $\Gamma$, the largest number of vertices such that the full subgraph of $\Gamma$ spanned by them is discrete.
    
We do not discuss the rather complicated proof here. We leave it an open question to find a corresponding result for the growth of the number of finite index Lie subalgebras of $L_\Gamma$. Indeed, we expect that these two series are closely related and that the latter is slightly easier to control than $(s_n(A_\Gamma))_{n\in\naturals}$.

We have identified $\gamma_{n,p}(A_\Gamma)$ with $\delta_{n,\finfield_p A_\Gamma}$ in Theorem~\ref{thm:cs2}. For a group $G$, we could define $\gamma_{n,p^e}$ as the subgroup generated by $\gamma_n$ and all $\gamma_i^{p^j}$ with $i p^j\ge n p^{e-1}$. When $G$ is free, it was shown by Lazard that $\gamma_{n,p^e}(G)$ coincides with the dimension subgroup $\delta_{n,\integers/p^e\integers[G]}$ while this does not hold for general $G$, see~\cite{Moran}.

We leave it as an exercise to extend Lazard's result to $A_\Gamma$.

\subsection{Homology gradients}
Given a group $G$ and a nested sequence of finite index normal subgroups $G_n\triangleleft G$
with $\bigcup_{n} \Gamma_n=\{1\}$, one defines for a field $\Bbbk$ the
\emph{$\Bbbk$-homology gradients}
\begin{equation*}
  b_i^{(2)}(G;\Bbbk) := \limsup_n \frac{b_i(G_n;\Bbbk)}{[G:G_n]}\text{ for }i\in\naturals.
\end{equation*}
For general groups $G$, it is unclear whether this quantity depends on the
particular chain $\{G_n\}$. Until recently, it was also unclear in which manner this
quantity depends on the coefficients $\Bbbk$. 
Avradmidi, Okun, and Schreve in \cite{AOS} use the classifying space $X_\Gamma$ and induced cell structures
for coverings to explicitly compute these homology gradients. Let $F_\Gamma$ be
the flag complex generated by $\Gamma$, i.e.~the largest simplicial complex
with vertex set $V$ and edge set $E$. Then
\begin{equation*}
  b_i^{(2)}(A_\Gamma;\Bbbk) = \overline{b_{i-1}}(F_\Gamma;\Bbbk)
\end{equation*}
where $\overline{b_*}(F_\Gamma;\Bbbk)$ denotes the dimension of the
reduced homology of $F_\Gamma$. In particular, for RAAGs the homology
gradient is independent of the chain of normal subgroups, even though
in many examples it does depend on the field of coefficients $\Bbbk$.

\section*{Acknowledgments}
We are deeply grateful to Jacques Darn\'e, Pierre de la Harpe and the
anonymous referee for their well-thought comments on preliminary
versions of our text.\\

\end{document}